\documentclass[reqno]{amsart}
\usepackage{amssymb}
\usepackage{hyperref}

\title[ Generalized Drazin-Riesz invertibility for operator matrices  ]
{  Generalized Drazin-Riesz invertibility for operator matrices   }

\author[ A. Tajmouati, M. Karmouni and S.Alaoui Chrifi ]
{  A. Tajmouati, M. Karmouni and S. Alaoui Chrifi}

\address{ A. Tajmouati, S. Alaoui Chrifi\, \newline
	Sidi Mohamed Ben Abdellah
	University,
	Faculty of Sciences Dhar Al Mahraz, Laboratory of Mathematical Analysis and Applications, Fez, Morocco.}
\email{abdelaziz.tajmouati@usmba.ac.ma}
\email{safae.alaouichrifi@usmba.ac.ma}

\address{M. Karmouni\newline
	Cadi Ayyad University, Multidisciplinary Faculty, Safi, Morocco.}
\email{mohammed.karmouni@uca.ma}


\newtheorem{theorem}{Theorem}[section]

\newtheorem{remark}{Remarks}
\newtheorem{lemma}{Lemma}[section]
\newtheorem{proposition}{Proposition}[section]
\newtheorem{corollary}{Corollary}[section]
\newtheorem{example}{Example}

\begin{document}
	\maketitle
	\begin{abstract}
	Let	$A\in\mathcal{B}(X)$, $B\in\mathcal{B}(Y)$ and $C\in\mathcal{B}(Y,X)$ where $X$ and $Y$ are infinite Banach or Hilbert spaces. Let $M_{C}=\begin{pmatrix}
	A & C \\
	0 & B \\
	\end{pmatrix}$ be $2\times 2$ upper triangular operator matrix acting on $X\oplus Y$. In this paper, we consider some necessary and sufficient conditions for $M_{C}$ to be generalized Drazin-Riesz invertible. Furthermore, the set $\bigcap_{C\in \mathcal{B}(Y,X)}\sigma_{gDR}(M_{C})$ will be investigated and their relation between $\bigcap_{C\in \mathcal{B}(Y,X)}\sigma_{b}(M_{C})$ will be studied, where $\sigma_{gDR}(M_{C})$ and $\sigma_{b}(M_{C})$ denote the generalized Drazin-Riesz spectrum and the Browder spectrum, respectively.
	\end{abstract}
	
\section{Introduction and preliminaries}
Let $X$ and $Y$ denote infinite dimensional complex Banach spaces and $\mathcal{B}(X,Y)$ denotes the algebra of all bounded linear operators from $X$ into $Y$. If $X=Y$ we write $\mathcal{B}(X)$ instead of $\mathcal{B}(X,X)$. For $T\in \mathcal{B}(X)$, we denote $\alpha(T)$ the dimension of the kernel $\mathcal{N}(T)$, $\beta(T)$ the codimension of the range $\mathcal{R}(T)$ and $\sigma(T)$ the spectrum of $T$.
An operator $T\in \mathcal{B}(X)$ is called bounded below if $T$ is one to one and $\mathcal{R}(T)$ is closed. Recall that an operator $T\in \mathcal{B}(X)$ is Kato if $\mathcal{R}(T)$ is closed and $\mathcal{N}(T)\subset \mathcal{R}(T^{n})$, $n\in \mathbb{N}$. Two important class of Kato operators is given by the class of bounded below operators, and the class of surjective operators.\\
We can define the set of upper semi-Fredholm operators, the set of lower semi-Fredholm operators, the set of left semi-Fredholm operators, the set of right semi-Fredholm operators and the set of Fredholm operators respectively as following:
\begin{align*}
	\Phi_{+}(X)& :=\{T\in \mathcal{B}(X): \alpha(T)<\infty \mbox{ and } \mathcal{R}(T) \mbox{ is closed }\};\\
	\Phi_{-}(X)& :=\{T\in \mathcal{B}(X): \beta(T)<\infty \mbox{ and } \mathcal{R}(T) \mbox{ is closed }\};\\
	\Phi_{le}(X)& :=\{T\in \Phi_{+}(X): \mathcal{R}(T) \mbox{ is a complemented subspace of X }\};\\
	\Phi_{re}(X)& :=\{T\in \Phi_{-}(X): \mathcal{N}(T) \mbox{ is a complemented subspace of X }\};\\
	\Phi(X)& :=\Phi_{+}(X)\cap \Phi_{-}(X).
	\end{align*}
The left semi-Fredholm spectrum is denoted by $\sigma_{le}(T)=\{\lambda\in \mathbb{C}: T-\lambda I\notin \Phi_{le}(X)\}$ and the right semi-Fredholm spectrum is denoted by $\sigma_{re}(T)=\{\lambda\in \mathbb{C}: T-\lambda I\notin \Phi_{re}(X)\}$.
If $T\in \mathcal{B}(X)$ is upper semi-Fredholm or lower semi-Fredholm then the index of $T$ is defined as $\mbox{ind}(T) :=\alpha(T)-\beta(T)$. Now, we sall recall that the ascent of an operator $T$ is given by $\mbox{asc}(T) :=\mbox{inf}\{n\in \mathbb{N}: \mathcal{N}(T^{n})=\mathcal{N}(T^{n+1})\}$ and the descent of $T$ is given by $\mbox{des}(T) :=\mbox{inf}\{n\in \mathbb{N}: \mathcal{R}(T^{n})=\mathcal{R}(T^{n+1})\}$, where $\mbox{inf}\>\emptyset = \infty$.\\
An operator $T\in \mathcal{B}(X)$ is upper semi-Browder if $T$ is upper semi-Fredholm and $\mbox{asc}(T)<\infty$. If $T$ is lower semi-Browder then $T$ is lower semi-Fredholm and $\mbox{des}(T)<\infty$. The operator $T\in \mathcal{B}(X)$ is Browder if $T$ is Fredholm with finite ascent and decent. The classes of operators defined above motivate the definition of the following spectra:
\begin{align*}
\sigma_{b_{+}}(T)&=\{\lambda\in \mathbb{C}: T-\lambda I \mbox{ is not upper semi-Browder }\};\\
\sigma_{b_{-}}(T)&=\{\lambda\in \mathbb{C}: T-\lambda I \mbox{ is not lower semi-Browder }\};\\
\sigma_{b}(T)&=\{\lambda\in \mathbb{C}: T-\lambda I \mbox{ is not Browder }\}=\sigma_{b_{+}}(T)\cup \sigma_{b_{-}}(T).
\end{align*}
Taking into account that $\sigma_{b_{+}}(T)=\sigma_{b_{-}}(T^{*})$ and $\sigma_{b_{-}}(T)=\sigma_{b_{+}}(T^{*})$ where $T^{*}$ denotes the dual of $T$. A relevant case is obtained if we assume that $T$ is left semi-Fredholm with finite ascent, in this case $T$ is called left semi-Browder. If $T$ is right semi-Fredholm operator with finite descent then $T$ is right semi-Browder. The left semi-Browder spectrum and the right semi-Browder spectrum of $T\in \mathcal{B}(X)$ are, respectively, denoted by:
\begin{align*}
\sigma_{lb}(T)&=\{\lambda \in \mathbb{C}: T-\lambda I \mbox{ is not left semi-Browder }\};\\
\sigma_{rb}(T)&=\{\lambda \in \mathbb{C}: T-\lambda I \mbox{ is not right semi-Browder }\}.
\end{align*}
Let $M$ be a subspace of $X$ for which $T(M)\subset M$, $T\in \mathcal{B}(X)$ then $M$ is called $T$-invariant, we denote by $T_{M}$ the restriction of $T$ to $M$. We say that $T$ is completely reduced by the pair $(M,N)$ and we denote $(M,N)\in \mbox{Red}(T)$, if $M$ and $N$ are two closed $T$-invariant subspaces of $X$ such that $X=M\oplus N$. Clearly, for $(M,N)\in \mbox{Red}(T)$, the operator $T\in \mathcal{B}(X)$ is of the form $T=T_{M}\oplus T_{N}$ and we say that $T$ is a direct sum of $T_{M}$ and $T_{N}$.

We shall say that an operator $T\in \mathcal{B}(X)$ is Riesz, if $T-\lambda I\in \Phi(X)$ for all $\lambda\in \mathbb{C}\backslash\{0\}$. In \cite[Theorem 3.111]{pa} it was proved that $T$ is Riesz if and only if  $T-\lambda I\in \Phi(X)$ and $\mbox{ind}(T-\lambda I)=0$ for every $\lambda\in \mathbb{C}\backslash\{0\}$. If there exists a non-zero complex polynomial $P(z)$ such that $P(T)$ is Riesz then we say that $T\in \mathcal{B}(X)$ is polynomially Riesz, these operators have been discussed in \cite{hll} and \cite{sddh}.\\
It is useful to mention that for $T\in \mathcal{B}(X)$ and $(M,N)\in\mbox{Red}(T)$ (\cite{sddh},\cite{zzl}):
$$T_{M} \mbox{ and } T_{N} \mbox{ are Browder } \Leftrightarrow T=T_{M}\oplus T_{N} \mbox{ is Browder };$$
$$T_{M} \mbox{ and } T_{N} \mbox{ are Riesz } \Leftrightarrow T=T_{M}\oplus T_{N} \mbox{ is Riesz }.$$

Recall that the concept of "generalized Kato-Riesz decomposition" was defined in \cite{szm}. Namely, an operator $T\in \mathcal{B}(X)$ admits a generalized Kato-Riesz decomposition (abbreviated GKRD) if there exists a pair $(M,N)\in \mbox{Red}(T)$ such that $T_{M}$ is Kato and $T_{N}$ is Riesz. Additionally, if we assume in the definition above that $T_{N}$ is quasi-nilpotent and $N$ is finite-dimensional then $T$ is called essentially Kato. In this case, it is easy to see that $T_{N}$ is nilpotent since every quasi-nilpotent operator on a finite dimensional space is nilpotent. It should be noted that the class of semi-Fredholm operators belong to the class of essentially Kato operators (see for instance, \cite[Theorem 16.21]{vm}.\\
The generalized Kato-Riesz spectrum as well as the essentially Kato spectrum are given by:
\begin{align*}
\sigma_{gKR}(T)&=\{\lambda\in \mathbb{C}: T-\lambda I \mbox{ does not admit a GKRD }\};\\
\sigma_{eK}(T)&=\{\lambda\in \mathbb{C}: T-\lambda I \mbox{ is not essentially Kato }\}.
\end{align*}
The class of generalized Drazin-Riesz invertible operators was first introduced by S\u{C}. \v{Z}ivkovi\'{c}-Zlatanovi\'{c} and M D. Cvetkovi\'{c} in \cite{szm} as follows:
an operator $T\in \mathcal{B}(X)$ is said to be generalized Drazin-Riesz invertible if there exists $S\in \mathcal{B}(X)$ such that
$$TS=ST,\>\>\>\>\>\>\> STS=S,\>\>\>\>\>\>\>   T-TST \mbox{ is Riesz. }$$
It is well known by \cite[Theorem 2.3]{szm} that $T\in \mathcal{B}(X)$ is generalized Drazin-Riesz invertible if and only if there exists $(M,N)\in \mbox{Red}(T)$ such that $T_{M}$ is Browder and $T_{N}$ is Riesz. Moreover, the definition of the following set can also be found in \cite{szm}
$$gDRM(X) :=\{T\in \mathcal{B}(X): T=T_{1}\oplus T_{2} \mbox{ where } T_{1} \mbox{ is bounded below and } T_{2} \mbox{ is Riesz }\};$$
$$gDRQ(X) :=\{T\in \mathcal{B}(X): T=T_{1}\oplus T_{2} \mbox{ where } T_{1} \mbox{ is surjective and } T_{2} \mbox{ is Riesz }\}.$$
Mainly, an operator $T\in \mathcal{B}(X)$ is called generalized Drazin-Riesz bounded below (\textit{resp}, generalized Drazin-Riesz surjective) if $T\in gDRM(X)$ (\textit{resp}, $T\in gDRQ(X)$).
The generalized Drazin-Riesz spectrum is given by
$$\sigma_{gDR}(T)=\{\lambda \in \mathbb{C}: T-\lambda I \mbox{ is not generalized Drzain-Riesz invertible}\}.$$
Among other things, Browder as well as Riesz operators are generalized Drazin-Riesz invertible. Furthermore, $\sigma_{gKR}(T)$ and $\sigma_{gDR}(T)$ are compact subsets of the complex plane $\mathbb{C}$, possibly empty and the generalized Drazin-Riesz spectrum of $T\in \mathcal{B}(X)$ is characterized as follows:
$$\sigma_{gDR}(T)=\sigma_{gKR}(T)\cup \mbox{acc}\sigma_{b}(T).$$

Hilbert space operators are also inserted among this paper. Accurately, $H$ and $K$ will be two infinite dimensional separable Hilbert spaces.\\
For $A\in \mathcal{B}(X)$, $B\in \mathcal{B}(Y)$ and $C\in \mathcal{B}(Y,X)$, we denote by $M_{C}$ the upper triangular operator matrix acting on $X\oplus Y$ of the form $M_{C}=\begin{pmatrix}
A & C \\
0 & B \\
\end{pmatrix}.$\\
In the recent past, the study of upper triangular operator matrices was a matter of great interest. More importantly, if $A\in \mathcal{B}(H)$ and $B\in \mathcal{B}(K)$, the sets $\bigcap_{C\in \mathcal{B}(K,H)}\sigma_{*}(M_{C})$ were considered by many authors (see, \cite{mi}, \cite{dsd}, \cite{slh}), where $\sigma_{*}$ runs over different kind of spectra. For instance, in \cite{slh} S.Zhang et al. have shown that
$$\bigcap_{C\in \mathcal{B}(K,H)}\sigma_{b}(M_{C})=\sigma_{b_{+}}(A)\cup \sigma_{b_{-}}(B)\cup W(A,B),$$
where $W(A,B)=\{\lambda \in \mathbb{C}: \alpha(A-\lambda I)+\alpha(B-\lambda I)\neq \beta(A-\lambda I)+\beta(B-\lambda I)\}$.

The ultimate goal of this paper is to provide sufficient conditions for which $M_{C}$ is generalized Drazin-Riesz invertible. In this regard, if $A\in \mathcal{B}(H)$, $B\in \mathcal{B}(K)$ and the following conditions hold:
\begin{enumerate}
	\item[(i)] $A$ and $B$ both admit a generalized Kato-Riesz decomposition;
	\item[(ii)] the upper semi-Browder spectrum does not cluster at $0$;
	\item[(iii)] the lower semi-Browder spectrum does not cluster at $0$;
	\item[(iv)] there exists $\delta>0$ such that $\alpha(A-\lambda I)+\alpha(B-\lambda I)= \beta(A-\lambda I)+\beta(B-\lambda I)$ for every $0<|\lambda|<\delta$.
\end{enumerate}
We prove then, the existence of $C\in \mathcal{B}(K,H)$ for which $M_{C}$ is generalized Drazin-Riesz invertible. Consequently, we obtain a description of $\bigcap_{C\in \mathcal{B}(K,H)}\sigma_{gDR}(M_{C})$. Additionally, we give necessary and sufficient condition for which $$\bigcap_{C\in \mathcal{B}(K,H)}\sigma_{gDR}(M_{C})=\bigcap_{C\in \mathcal{B}(K,H)}\sigma_{b}(M_{C}).$$

\section{Main results}	
The following lemmas are among the most wiedly used results of this paper. The first lemma is an overview of the punctured neighborhood theorem (\cite[Theorem 18.7]{vm} and \cite[page 35]{pa}).
\begin{lemma}\cite{vm}\label{pt}
	If $T\in \Phi_{+}(X)\cup \Phi_{-}(X)$. Then there exists a constant $\varepsilon>0$ such that for all $\lambda\in D(0,\varepsilon)\backslash\{0\}$, $\alpha(\lambda I-T)$ and $\beta(\lambda I-T)$ are constant. Moreover, $\mbox{ind}(\lambda I-T)= \mbox{ind}(T)$ for every $|\lambda|<\varepsilon$.	
\end{lemma}	

\begin{lemma}\cite{szm}\label{mn}
	Let $T\in \mathcal{B}(X)$.
	\begin{enumerate}
		\item[(a)] The following assertions are equivalent:
	\begin{enumerate}
		\item[(i)] $T$ admits a GKRD and $0\notin \mbox{acc}\sigma_{b_{+}}(T)$,
		\item[(ii)] $T=T_{M}\oplus T_{N}$ with $T_{M}$ is upper semi-Browder and $T_{N}$ is Riesz.
	\end{enumerate}
	    \item[(b)]The following assertions are equivalent:
	    \begin{enumerate}
	    	\item[(i)] $T$ admits a GKRD and $0\notin \mbox{acc}\sigma_{b_{-}}(T)$,
	    	\item[(ii)] $T=T_{M}\oplus T_{N}$ with $T_{M}$ is lower semi-Browder and $T_{N}$ is Riesz.
		\end{enumerate}
	\end{enumerate}
\end{lemma}

	\begin{lemma}\label{L2}
		Let $A\in \mathcal{B}(X)$, $B\in \mathcal{B}(Y)$ and $C\in \mathcal{B}(Y,X)$. If $M_{C}$ is generalized Drazin-Riesz invertible, then the following statements hold:
		\begin{enumerate}
			\item[(i)] $0\notin \mbox{acc}\sigma_{lb}(A)$,
			\item[(ii)] $0\notin \mbox{acc}\sigma_{rb}(B)$,
			\item[(iii)] there exists $\delta>0$ such that $\alpha(\lambda I-A)+\alpha (\lambda I-B)=\beta(\lambda I-A)+ \beta(\lambda I-B)$ for every $0<|\lambda|<\delta$.
		\end{enumerate}
	\end{lemma}
	\begin{proof}
		Certainly, $\sigma_{gDR}(M_{C})=\sigma_{gKR}(M_{C})\cup \mbox{acc}\sigma_{b}(M_{C}).$ Since $M_{C}$ is generalized Drazin-Riesz invertible then $0\notin \sigma_{gDR}(M_{C})$, it follows that $0\notin \mbox{acc}\sigma_{b}(M_{C})$, which implying that there exists $\delta>0$ such that $M_{C}-\lambda I$ is Browder for every $0<|\lambda|<\delta$. According to \cite[ Theorem 2.9]{slh} $A-\lambda I$ is left semi-Browder $B-\lambda I$ is right semi-Browder and $\alpha(\lambda I-A)+\alpha (\lambda I-B)=\beta(\lambda I-A)+ \beta(\lambda I-B)$ for every $0<|\lambda|<\delta$. This proved that $0\notin \mbox{acc}\sigma_{lb}(A)$, $0\notin \mbox{acc}\sigma_{rb}(B)$ and $\alpha(\lambda I-A)+\alpha (\lambda I-B)=\beta(\lambda I-A)+ \beta(\lambda I-B)$ for every $0<|\lambda|<\delta$.
		\end{proof}
		
		\begin{proposition}
			Let $A\in\mathcal{B}(X)$, $B\in\mathcal{B}(Y)$ and $C\in\mathcal{B}(Y,X)$. If $M_{C}$ is generalized Drazin-Riesz invertible with a Drazin-Riesz inverse $S$ of the form $\begin{pmatrix}
				U & W \\
				0 & V \\
			\end{pmatrix}$ then $A$ and $B$ are generalized Drazin-Riesz invertible.
			\end{proposition}
			\begin{proof}
				Let $S=\begin{pmatrix}
				U & W \\
				0 & V \\
				\end{pmatrix}$ be the generalized Drazin-Riesz inverse of $M_{C}$. It is easily seen that $AU=UA$, $BV=VB$, $UAU=U$ and $VBV=V$ since $M_{C}S=SM_{C}$ and $SM_{C}S=S$. On the other hand
				$$M_{C}-M_{C}SM_{C}=\begin{pmatrix}
				A-AUA & L \\
				0 & B-BVB \\
				\end{pmatrix}$$
				where $L=C-AUC-AWB-CVB$, we have $M_{C}-M_{C}SM_{C}$ is Riesz and hence, $\begin{pmatrix}
				A-AUA-\lambda I & L \\
				0 & B-BVB-\lambda I \\
				\end{pmatrix}$ is Fredholm for every $\lambda\in \mathbb{C}\backslash\{0\}$. Then by \cite[Theorem 3.2]{dsd}, $\sigma_{le}(A-AUA)=\{0\}$ and $\sigma_{re}(B-BVB)=\{0\}$. Therefore, using \cite[Theorem 2.3]{sddh} we have $\sigma_{e}(A-AUA)=\{0\}$ and $\sigma_{e}(B-BVB)=\{0\}$. As a result, $A$ and $B$ are generalized Drazin-Riesz invertible.
				\end{proof}
			
		This theorem outlines sufficient conditions for which $M_{C}$ is generalized Drazin-Riesz invertible for some $C\in \mathcal{B}(K,H)$.
		\begin{theorem}\label{Th1}
			Let $A\in \mathcal{B}(H)$ and $B\in \mathcal{B}(K)$ such that the following statements hold:
			\begin{enumerate}
				\item[(i)] $A$ and $B$ both admit a generalized Kato-Riesz decomposition,
				\item[(ii)] $0\notin \mbox{acc}\sigma_{b_{+}}(A)$,
				\item[(iii)] $0\notin \mbox{acc}\sigma_{b_{-}}(B)$,
				\item[(iv)] there exists $\delta>0$ such that $\alpha(\lambda I-A)+\alpha (\lambda I-B)=\beta(\lambda I-A)+ \beta(\lambda I-B)$ for every $0<|\lambda|<\delta$.
			\end{enumerate}
			Then $M_{C}$ is generalized Drazin-Riesz invertible for some $C\in \mathcal{B}(K,H)$.
		\end{theorem}
	\begin{proof}
		By means of lemma \ref{mn}, there exists a pair $(H_{1},H_{2})\in \mbox{Red}(A)$ such that $A_{H_{1}}=A_{1}$ is upper semi-Browder and $A_{H_{2}}=A_{2}$ is Riesz. Also, there exists a pair $(K_{1},K_{2})\in \mbox{Red}(B)$ such that $B_{K_{1}}=B_{1}$ is lower semi-Browder and $B_{K_{2}}=B_{2}$ is Riesz. It is clear that for each $\lambda\in \mathbb{C}\backslash\{0\}$, $\mbox{ind}(A_{2}-\lambda I)=0$ and $\mbox{ind}(B_{2}-\lambda I)=0$. Furthermore, $A_{1}$ is upper semi-Fredholm and $B_{1}$ is lower semi-Fredholm.
	    So according to lemma \ref{pt} (\textit{punctured neighborhood theorem}) there exists $\varepsilon>0$ such that $\mbox{ind}(A_{1})=\mbox{ind}(A_{1}-\lambda I)$ and $\mbox{ind}(B_{1})=\mbox{ind}(B_{1}-\lambda I)$ for every $|\lambda|<0$.
		It is straightforward that
		$$A-\lambda I=\begin{pmatrix}
		A_{1}-\lambda I & 0 \\
		0 & A_{2}-\lambda I \\
		\end{pmatrix} \mbox{  and    } B-\lambda I=\begin{pmatrix}
		B_{1}-\lambda I & 0 \\
		0 & B_{2}-\lambda I \\
		\end{pmatrix}.$$
		Thus, for each $0<|\lambda|<\varepsilon$
			\begin{equation}\label{eq1}
			 \mbox{ind}(A-\lambda I) = \mbox{ind}(A_{1}-\lambda I)
			 \hspace{0,2cm}\textit{ and }\hspace{0,2cm}\mbox{ind}(B-\lambda I) = \mbox{ind}(B_{1}-\lambda I).
			 \end{equation}
		Combining \ref{eq1} with the statement (iv), we have for $\lambda_{0}\in \mathbb{C}$, $0<|\lambda_{0}|<\mbox{min}\{\varepsilon,\delta\}$.
		\begin{align*}
			\mbox{ind}(A_{1}) &= \mbox{ind}(A_{1}-\lambda_{0} I) = \mbox{ind}(A-\lambda_{0} I) = \alpha(A-\lambda_{0} I)-\beta(A-\lambda_{0} I)\\
			& = -(\alpha(B-\lambda_{0} I)-\beta(B-\lambda_{0} I))= - \mbox{ind}(B-\lambda_{0} I)\\
			&=- \mbox{ind}(B_{1}-\lambda_{0} I)=- \mbox{ind}(B_{1}).
			\end{align*}
		From \cite[Theorem 2.9]{slh} it follows that there exists $C_{1}\in \mathcal{B}(K_{1},H_{1})$ such that:
		$$\begin{pmatrix}
		A_{1} & C_{1} \\
		0 & B_{1} \\
		\end{pmatrix}
		: H_{1}\oplus K_{1}\rightarrow H_{1}\oplus K_{1} \mbox{ is Browder. }$$
	    Define $C\in \mathcal{B}(K,H)$ as follows $\begin{pmatrix}
		C_{1} & 0 \\
		0 & 0 \\
		\end{pmatrix}
		: K_{1}\oplus K_{2}\rightarrow H_{1}\oplus H_{2}$. It is easy to show that $H_{1}\oplus K_{1}$ and $H_{2}\oplus K_{2}$ are closed subspaces of $H\oplus K$ invariant by $M_C=\begin{pmatrix}
		A & C \\
		0 & B \\
		\end{pmatrix}$. Moreover,
		$${M_{C_{H_{1}\oplus K_{1}}}}=\begin{pmatrix}
		A_{1} & C_{1} \\
		0 & B_{1} \\
		\end{pmatrix} \mbox{ is Browder }
		\mbox{ and }{M_{C_{H_{2}\oplus K_{2}}}}=\begin{pmatrix}
		A_{2} & 0 \\
		0 & B_{2} \\
		\end{pmatrix} \mbox{ is Riesz }.$$
		Hence $M_C=\begin{pmatrix}
		A_{1} & C_{1} & 0 & 0 \\
		0 & B_{1} & 0 & 0 \\
		0 & 0 & A_{2} & 0 \\
		0 & 0 & 0 & B_{2}\\
		\end{pmatrix}: H_{1}\oplus K_{1}\oplus H_{2}\oplus K_{2}\rightarrow H_{1}\oplus K_{1}\oplus H_{2}\oplus K_{2}$ is generalized Drazin-Riesz invertible.
	
		\end{proof}
	\begin{corollary}
		\begin{enumerate}
			\item[(i)] Let $A\in \mathcal{B}(X)$ and $B\in \mathcal{B}(Y)$. Then:
			$$\mbox{acc}\sigma_{lb}(A)\cup \mbox{acc}\sigma_{rb}(B)\cup \mathcal{W}\subset \bigcap_{C\in \mathcal{B}(Y,X)}\sigma_{gDR}(M_{C}),$$
			where $\mathcal{W}=\{\lambda\in \mathbb{C}:\nexists \delta>0 \mbox{ such that } \alpha(A-\lambda I-\lambda'I)+\alpha(B-\lambda I-\lambda'I)=\beta(A-\lambda I-\lambda'I)+\beta(B-\lambda I-\lambda'I) \mbox{ for each } 0<|\lambda'|<\delta\}.$
			\item[(ii)] Let $A\in \mathcal{B}(H)$ and $B\in \mathcal{B}(K)$. Then:
			\begin{align*}
			\mbox{acc}\sigma_{b_{+}}(A)& \cup \mbox{acc}\sigma_{b_{-}}(B)\cup \mathcal{W}  \subset \bigcap_{C\in \mathcal{B}(K,H)}\sigma_{gDR}(M_{C})\\
			& \subset \mbox{acc}\sigma_{b_{+}}(A)\cup \mbox{acc}\sigma_{b_{-}}(B)\cup \mathcal{W}\cup\sigma_{gKR}(A) \cup\sigma_{gKR}(B).
			\end{align*}
			In particular, if $\sigma_{gKR}(A)\subseteq \mbox{acc}\sigma_{b_{+}}(A)$ and $\sigma_{gKR}(B)\subseteq \mbox{acc}\sigma_{b_{-}}(B)$ then
			$$\bigcap_{C\in \mathcal{B}(K,H)}\sigma_{gDR}(M_{C})=\mbox{acc}\sigma_{b_{+}}(A)\cup \mbox{acc}\sigma_{b_{-}}(B)\cup \mathcal{W}.$$
		\end{enumerate}
	\end{corollary}
	
	\begin{proof}
		Due to lemma \ref{L2} and theorem \ref{Th1} .
		\end{proof}
		
		\begin{remark}\label{rek}
			Let $T\in \mathcal{B}(X)$
			\begin{enumerate}
				\item[(a)] If $T$ is either upper semi-Browder or lower semi-Browder then $T$ is essentially Kato and hence admits a GKRD. Consequently, the following inclusions hold:
				$$\sigma_{gKR}(T)\subseteq \sigma_{eK}(T)\subseteq \sigma_{b_{+}}(T)\cap\sigma_{b_{-}}(T).$$ It follows that $\sigma_{gKR}(T)\subseteq  acc\sigma_{b_{+}}(T)\cap acc\sigma_{b_{-}}(T)$ when $\sigma_{b_{+}}(T)= acc\sigma_{b_{+}}(T)$ and $\sigma_{b_{-}}(T)= acc \sigma_{b_{-}}(T)$.
				\item[(b)] If $T$ is polynomially Riesz then $\sigma_{gKR}(T)=acc\sigma_{b_{+}}(T)= acc\sigma_{b_{-}}(T)=\emptyset$.
			\end{enumerate}
		\end{remark}
		The following example is coted from \cite{mi}, it is also valid to build a particular operator matrix who satisfies the conditions of theorem \ref{Th1} .
		\begin{example}
			$U$ and $V$ are the forward and the backward unilateral shifts on $l^{2}(\mathbb{N})$, we have:
			\begin{align*}
				\sigma_{b_{+}}(U) & = acc \sigma_{b_{+}}(U)= \partial \mathbb{D},\\
				\sigma_{b_{-}}(V) & = acc \sigma_{b_{-}}(V)= \partial \mathbb{D},
				\end{align*}
				where $\mathbb{D}=\{\lambda\in \mathbb{C}: |\lambda|\leqslant 1\}$. It is known that $U$ is upper semi-Browder, $V$ is lower semi-Browder with $\mbox{asc}(U)=\mbox{des}(V)=1$, $\mbox{ind}(U)=-1$ and $\mbox{ind}(V)=1$. Following remark \ref{rek}, $U$ and $V$ admit a GKRD.\\
				Now let $W\in \mathcal{B}(l^{2}(\mathbb{N})\oplus l^{2}(\mathbb{N}))$ defined by $W=\begin{pmatrix}
				V & 0 \\
				0 & 0 \\
				\end{pmatrix}$. Then $\sigma_{b_{-}}(W)=\partial \mathbb{D} \cup \{0\}$ and $0\notin acc \sigma_{b_{-}}(W)$, but $W$ admits a GKRD.\\
				On the other hand, from lemma \ref{pt} there exists $\delta>0$ such that $\mbox{ind}(U-\lambda I)=\mbox{ind}(U)$ and $\mbox{ind}(V-\lambda I)=\mbox{ind}(V)$ for $0<|\lambda|<\delta$, then:
				$$\mbox{ind}(U-\lambda I)=\mbox{ind}(U)=-\mbox{ind}(V)=-\mbox{ind}(V-\lambda I)=-\mbox{ind}(W-\lambda I),$$
				for every $0<|\lambda|<\delta$. Thus by theorem \ref{Th1} there exists $C\in \mathcal{B}(l^{2}(\mathbb{N})\oplus l^{2}(\mathbb{N}),l^{2}(\mathbb{N}))$ such that:
				$$
				M_C=\begin{pmatrix}
				U & C \\
				0 & W \\
				\end{pmatrix}
				: l^{2}\oplus l^{2}\oplus l^{2}\rightarrow l^{2}\oplus l^{2}\oplus l^{2},
				$$
				is generalized Drazin-Riesz invertible, i.e. $0\notin \bigcap_{C\in \mathcal{B}(l^{2}(\mathbb{N})\oplus l^{2}(\mathbb{N}),l^{2}(\mathbb{N}))}\sigma_{gDR}(M_{C})$.
		\end{example}
	
	In connection with Theorem \ref{Th1}, it is straightforward to obtain the following result.
	
	\begin{proposition}
		Let $A\in \mathcal{B}(H)$ and $B\in \mathcal{B}(K)$ such that the following statments hold:
		\begin{enumerate}
			\item[(i)] $A$ is generalized Drazin-Riesz bounded below,
			\item[(ii)] $B$ is generalized Drazin-Riesz surjective,
			\item[(iii)] there exists $\delta>0$ such that $\alpha(\lambda I-A)+\alpha (\lambda I-B)=\beta(\lambda I-A)+ \beta(\lambda I-B)$ for every $0<|\lambda|<\delta$.
		\end{enumerate}
		Then there exists $C\in \mathcal{B}(K,H)$ such that $M_{C}$ is generalized Drazin-Riesz invertible.
	\end{proposition}
	
	\begin{proof}
		Because of \cite[Theorem 2.4]{szm} and \cite[Theorem 2.5]{szm} the following equivalences hold:
	\begin{align*}
		& A\in gDRM(H) \mbox{ and } B\in gDRQ(K)\\
		& \Longleftrightarrow A \mbox{ and } B \mbox{ both admit a GKRD}, 0\notin \mbox{acc}\sigma_{b_{+}}(T)
		\mbox{ and } 0\notin \mbox{acc}\sigma_{b_{-}}(T);\\
		& \Longleftrightarrow \mbox{ there exist } (H_{1},H_{2})\in\mbox{Red}(A), (K_{1},K_{2})\in\mbox{Red}(B)
		\mbox{ such that } A_{1}=A_{H_{1}} \mbox{is upper}\\
		& \mbox{ semi-Browder}, B_{1}=A_{K_{1}} \mbox{is lower semi-Browder},
		A_{2}=A_{H_{2}} \mbox{ and } B_{2}=A_{K_{2}}
		\mbox{are Riesz.}
		\end{align*}	
		Therefore, by the proof of Theorem \ref{Th1} there exists $C\in \mathcal{B}(K,H)$ such that: $$M_{C}=(M_{C})_{H_{1}\oplus K_{1}}\oplus(M_{C})_{H_{2}\oplus K_{2}}.$$
		$$\mbox{ Where, } (M_{C})_{H_{1}\oplus K_{1}} \mbox{ is Browder and }(M_{C})_{H_{2}\oplus K_{2}} \mbox{ is Riesz }.$$
		As a result, $M_{C}$ is generalized Drazin-Riesz invertible.

		\end{proof}

	Let $K\subseteq \mathbb{C}$ be a compact set. Denote by $\eta K$ the connected hull of $K$, where the complement $\mathbb{C}\backslash \eta K$ is the unique unbounded component of the complement $\mathbb{C}\backslash K$ (\cite[Definition 7.10.1]{his}), a hole of $K$ is a component of $\eta K\backslash K$.\\
	If $H,K\subseteq \mathbb{C}$ are compact subsets, we get (\cite[Theorem 7.10.3]{his}):
	 $$\partial H\subseteq K\subseteq H \Longrightarrow \partial H\subseteq \partial K \subseteq K \subseteq H \subseteq \eta K= \eta H.$$
	 For instance, if $K\subseteq \mathbb{C}$ is finite then $\eta K= K$. Moreover the following implication holds:	
	 $$ \eta K= \eta H \Longrightarrow (\mbox{H finite} \Leftrightarrow \mbox{K finite}), $$
	and in this case $H=K$.

		\begin{theorem}
			For $A\in \mathcal{B}(X)$, $B\in \mathcal{B}(Y)$ and $C\in \mathcal{B}(Y,X)$. Then the following statements are equivalent:
			\begin{enumerate}
				\item[(i)] $\sigma_{gKR}(M_{C})=\emptyset$,
				\item[(ii)] $\sigma_{gDR}(M_{C})=\emptyset$,
				\item[(iii)] $A$ and $B$ are polynomially Riesz,
				\item[(iv)] $\sigma_{b}(M_{C})$ is finite.
			\end{enumerate}
		\end{theorem}
		
		\begin{proof}
				$(i)\Leftrightarrow (ii)$ From \cite[Theorem 3.10]{szm}, it follows that $\eta \sigma_{gKR}(M_{C})= \eta \sigma_{gDR}(M_{C})$. This implies that $\sigma_{gKR}(M_{C})=\emptyset$ if and only if $\sigma_{gDR}(M_{C})=\emptyset$.\\
				$(ii)\Rightarrow (iv)$ assume that $\sigma_{gDR}(M_{C})=\emptyset$, since $\sigma_{gDR}(M_{C})=\sigma_{gKR}(M_{C})\cup \mbox{acc}\sigma_{b}(M_{C})$. then $\mbox{acc}\sigma_{b}(M_{C})=\emptyset$. Consequently, $\sigma_{b}(M_{C})$ is a finite set.\\
				$(iii)\Rightarrow (ii)$ Let $A$ and $B$ be polynomially Riesz, then $\sigma_{b}(A)=\pi_{A}^{-1}(0)=\{\lambda_{1},\lambda_{2},...,\lambda_{n}\}$ and $\sigma_{b}(B)=\pi_{B}^{-1}(0)=\{\mu_{1},\mu_{2},...,\mu_{m}\}$ where $\pi_{A}$ and $\pi_{B}$ are the minimal polynomials of $A$ and $B$. According to the proof of \cite[Theorem 2.6]{zzl} we get, $$\eta \sigma_{b}(M_{C})= \eta (\sigma_{b}(A)\cup \sigma_{b}(B)).$$
				Also, it is known that $ \sigma_{b}(M_{C})\subset \sigma_{b}(A)\cup \sigma_{b}(B)$. Since $ \sigma_{b}(M_{C})$ and $\sigma_{b}(A)\cup \sigma_{b}(B)$ are finite then, $ \sigma_{b}(M_{C})= \sigma_{b}(A)\cup \sigma_{b}(B)=\pi_{A}^{-1}(0)\cup \pi_{B}^{-1}(0)=\{\lambda_{1},\lambda_{2},...,\lambda_{n},\mu_{1},\mu_{2},...,\mu_{m}\}.$ Thus, for every $\lambda\notin \pi_{A}^{-1}(0)\cup \pi_{B}^{-1}(0)$, $M_{C}-\lambda$ is Browder and hence generalized Drazin-Riesz invertible.\\
				Now from \cite[Theorem 2.13]{sddh} the Banach spaces $X$ and $Y$ are decomposed into the direct sums $X=X_{1}\oplus X_{2}\oplus ... \oplus X_{n},$ $Y=Y_{1}\oplus Y_{2}\oplus ... \oplus Y_{m}$ where $X_{i}$ is closed $A$-invariant subspace of $X$ (\textit{resp}, $Y_{j}$ is closed $B$-invariant subspace of $Y$). $A=A_{1}\oplus A_{2} \oplus ... \oplus A_{n}$ and $B=B_{1}\oplus B_{2} \oplus ... \oplus B_{m}$, $A_{i}=A_{X_{i}}$ and $B_{j}=B_{Y_{j}}$.\\
				$A_{i}-\lambda_{i}$ and $B_{j}-\mu_{j}$ are Riesz which implies that $\sigma_{b}(A_{i})\subset\{\lambda_{i}\}$ and $\sigma_{b}(B_{j})\subset\{\mu_{j}\}$ for i=1,...,n ; j=1,...,m.\\
				\textbf{Case 1:} If $n=m$, then $A_{i}-\lambda_{j}$ and $B_{i}-\mu_{j}$ are Browder for every $i\neq j$, $i,j \in \{1,...,n\}.$ It is clear that $X_{1}\oplus Y_{1}$, $X_{2}\oplus Y_{2}$, ... ,$X_{n}\oplus Y_{n}$ are closed subspace of $X\oplus Y$ invariant by $M_{C}$. Moreover, $\sigma_{b}({M_{C}}_{X_{i}\oplus Y_{i}})=\sigma_{b}(A_{i})\cup \sigma_{b}(B_{i})\subset\{\lambda_{i},\mu_{i}\}$ for $i=1,...,n$. Using the decomposition:
				$$M_{C}-\lambda_{1}={M_{C}}_{X_{1}\oplus Y_{1}}-\lambda_{1}\oplus {M_{C}}_{X_{2}\oplus Y_{2}}-\lambda_{1}\oplus...\oplus{M_{C}}_{X_{n}\oplus Y_{n}}-\lambda_{1}.$$
				We have $M_{C}-\lambda_{1}$ is Browder, because ${M_{C}}_{X_{i}\oplus Y_{i}}-\lambda_{1}$ are Browder for every $i\in\{1,2,...,n\}$. It follows that $M_{C}-\lambda_{1}$ is generalized Drazin-Riesz invertible. In the same way, we get that $M_{C}-\lambda_{i}$ and $M_{C}-\mu_{i}$ are generalized Drazin-Riesz invertible for every $i\in \{1,...,n\}$.\\
				\textbf{Case 2:} If $n\neq m$ we will consider only the case where $n>m$ since the argument of the case $n<m$ is similar. So, $M_{C}-\lambda_{i}$ and $M_{C}-\mu_{i}$ are Browder thus, $M_{C}-\lambda_{i}$ and $M_{C}-\mu_{i}$ are generalized Drazin-Riesz invertible for every $i\in \{1,...,m\}$.
				For $i=m+1$, we consider the decomposition:
				\begin{align*}
					M_{C}-\lambda_{m+1}={M_{C}}_{X_{1}\oplus Y_{1}}-\lambda_{m+1}& \oplus {M_{C}}_{X_{2}\oplus Y_{2}}-\lambda_{m+1}\oplus...\oplus{M_{C}}_{X_{m}\oplus Y_{m}}-\lambda_{m+1} \\
					& \oplus{M_{C}}_{X_{m+1}}-\lambda_{m+1}\oplus ...\oplus{M_{C}}_{X_{n}}-\lambda_{m+1} .
					\end{align*}
			   Since ${M_{C}}_{X_{m+1}}-\lambda_{m+1}$ is Riesz and ${M_{C}}_{X_{1}\oplus Y_{1}}-\lambda_{m+1} \oplus {M_{C}}_{X_{2}\oplus Y_{2}}-\lambda_{m+1}\oplus...\oplus{M_{C}}_{X_{m}\oplus Y_{m}}-\lambda_{m+1}
			    \oplus{M_{C}}_{X_{m+2}}-\lambda_{m+1}\oplus ...\oplus{M_{C}}_{X_{n}}-\lambda_{m+1}$
				is Browder. As a result, $M_{C}-\lambda_{m+1}$ is generalized Drazin-Riesz invertible. Similarly, we can prove that $M_{C}-\lambda_{k}$ is generalized Drazin-Riesz invertible for every $k\in\{m+2,...,n\}$.\\
				Which shows that $M_{C}-\lambda$ is generalized Drazin-Riesz invertible for every $\lambda\in \mathbb{C}$. consequently, $\sigma_{gDR}(M_{C})=\emptyset$.\\
				$(iii)\Rightarrow(iv)$ obvious.\\
				$(iv)\Rightarrow(iii)$ suppose that $\sigma_{b}(M_{C})$ is finite, we have
				$$\sigma_{b}(M_{C})\subseteq\sigma_{b}(A)\cup\sigma_{b}(B)\subseteq\eta(\sigma_{b}(A)\cup\sigma_{b}(B))=\eta(\sigma_{b}(M_{C}))=\sigma_{b}(M_{C}).$$
				Thus $\sigma_{b}(M_{C})=\sigma_{b}(A)\cup\sigma_{b}(B)$ which implies that $\sigma_{b}(A)$and $\sigma_{b}(B)$ are finite. Then according to \cite[Theorem 2.3]{sddh} $A$ and $B$ are polynomially Riesz.
		
			\end{proof}
			
			\begin{lemma}\label{L1}
				Let $T\in \mathcal{B}(X)$.
				\begin{enumerate}
					\item[(i)] $T$ is essentially Kato and $0\notin \mbox{acc}\sigma_{b_{+}}(T)$ if and only if $T$ is upper semi-Browder.
					\item[(ii)] $T$ is essentially Kato and $0\notin \mbox{acc}\sigma_{b_{-}}(T)$ if and only if $T$ is lower semi-Browder.	
				\end{enumerate}
			\end{lemma}
			
			\begin{proof}
				$(i)$ If $T$ is upper semi Browder then $T$ is essentially Kato. Since $0\notin \sigma_{b_{_{+}}}(T)$ then $0\notin \mbox{acc}\sigma_{b_{+}}(T)$.\\
				Conversely, Suppose that $T$ is essentially Kato and $0\notin \mbox{acc}\sigma_{b_{+}}(T)$ then, by \cite[Theorem 2.2]{qh}, there exists $\delta>0$ such that $T-\lambda$ is Kato for each $\lambda$ such that $0<|\lambda|<\delta$. Since $0\notin \mbox{acc}\sigma_{b_{+}}(T)$ there exists $\delta'>0$ such that $T-\lambda$ is upper semi-Browder for every $0<|\lambda|<\delta'$. For the seek of simplicity we tend to assume that $\delta=\delta'$. Hence, from \cite[Lemma 20.9]{vm}, $T-\lambda$ is bounded below for every $0<|\lambda|<\delta$, which means that $0\notin \mbox{acc}\sigma_{ap}(T)$. Moreover, let $(M,N)\in \mbox{Red}(T)$ where $T_{M}$ is Kato, $T_{N}$ is nilpotent and $N$ is finite dimensional. Clearly,
				$$\sigma_{ap}(T)=\sigma_{ap}(T_{M})\cup \sigma_{ap}(T_{N})=\sigma_{ap}(T_{M})\cup \{0\}.$$
				As $0\notin \mbox{acc}\sigma_{ap}(T)$, it follows that $0\notin \mbox{acc}\sigma_{ap}(T_{M})$ also $T_{M}$ is Kato then according to \cite[theorem 2.14]{am} $T_{M}$ is bounded below. Altogether, $T$ is upper semi-Browder \cite[Theorem 20.10]{vm}.\\
				$(ii)$ Suppose that $T$ is essentially Kato then, by \cite[Theorem 21.5]{vm}, $T^{*}$ is essentially Kato. Also, we have $0\notin \mbox{acc}\sigma_{b_{-}}(T)=\mbox{acc}\sigma_{b_{+}}(T^{*})$ and hence, by the first part $T^{*}$ is upper semi-Browder, or equivalently $T$ is lower semi-Browder.\\
				The reverse implication is obvious.
			\end{proof}
		
		Under further hypothesis the converse of theorem \ref{Th1} is also hold.
		
		\begin{proposition}
			Let $A\in \mathcal{B}(H)$ and $B\in \mathcal{B}(K)$ be essentially Kato operators, if $M_C$ is generalized Drazin-Riesz invertible for some $C\in \mathcal{B}(K,H)$. Then $A$ is upper semi-Browder, $B$ is lower semi-Browder and $\alpha(A)+\alpha(B)=\beta(A)+\beta(B)$ i.e. $M_{C_{1}}$ is Browder for some $C_{1}\in \mathcal{B}(K,H)$.
		\end{proposition}
		
	\begin{proof}
		It follows from lemma \ref{L2} that $0\notin \mbox{acc}\sigma_{b_{+}}(A)\cup \mbox{acc}\sigma_{b_{-}}(B)$ and there exists $\delta>0$ such that $\alpha(\lambda I-A)+\alpha (\lambda I-B)=\beta(\lambda I-A)+ \beta(\lambda I-B)$ for every $0<|\lambda|<\delta$. In addition to this $A$ is upper semi-Browder and $B$ is lower semi-Browder according to lemme \ref{L1}. Also, by lemma \ref{pt} it is easy to see that $\alpha(A)+\alpha(B)=\beta(A)+\beta(B)$. Thus, following \cite[Theorem 2.9]{slh} there exists $C_{1}\in \mathcal{B}(K,H)$ such that $M_{C_{1}}$ is Browder.
		\end{proof}
		
		\begin{corollary}
		Let $A\in \mathcal{B}(H)$ and $B\in \mathcal{B}(K)$. Then:
		$$\sigma_{eK}(A) \cup\sigma_{eK}(B)\cup \bigcap_{C\in \mathcal{B}(K,H)}\sigma_{gDR}(M_{C})= \bigcap_{C\in \mathcal{B}(K,H)}\sigma_{b}(M_{C}).$$
		\end{corollary}
		
		\begin{corollary}
		Let $A\in \mathcal{B}(H)$ and $B\in \mathcal{B}(K)$. Then the following statements are equivalent:
		\begin{enumerate}
			\item[(i)] $\bigcap_{C\in \mathcal{B}(K,H)}\sigma_{gDR}(M_{C})= \bigcap_{C\in \mathcal{B}(K,H)}\sigma_{b}(M_{C})$,
			\item[(ii)] $\sigma_{eK}(A) \cup\sigma_{eK}(B)\subset \bigcap_{C\in \mathcal{B}(K,H)}\sigma_{gDR}(M_{C})$.
		\end{enumerate}
		\end{corollary}

\end{document}